\documentclass[11pt,a4paper,reqno]{amsart}
\usepackage[a4paper,margin=28mm]{geometry}
\usepackage[T1]{fontenc}
\usepackage{lmodern}
\usepackage{amsmath,amssymb,amsthm}
\usepackage{microtype}
\usepackage{needspace}
\usepackage[colorlinks=true,linkcolor=blue,citecolor=blue,urlcolor=blue]{hyperref}
\newtheorem{maintheorem}{Theorem}

\newtheorem{maincorollary}[maintheorem]{Corollary}
\newtheorem{theorem}{Theorem}[section]
\newtheorem{lemma}[theorem]{Lemma}
\newtheorem{proposition}[theorem]{Proposition}

\theoremstyle{remark}
\newtheorem{remark}[theorem]{Remark}
\numberwithin{equation}{section}
\newcommand{\N}{\mathbb N}
\newcommand{\C}{\mathbb C}
\newcommand{\E}{\mathcal E}
\newcommand{\M}{\mathcal M}
\newcommand{\norm}[1]{\lVert #1\rVert}
\newcommand{\abs}[1]{\lvert #1\rvert}

\allowdisplaybreaks[1]
\hypersetup{pdftitle={A solution to Michael's problem for Fr\'echet algebras},pdfauthor={Antonio Acuaviva},pdfsubject={Boundedness of characters, automatic continuity, and Oka approximation}}

\begin{document}
\title[Michael's problem for Fr\'echet algebras]{A solution to Michael's problem for Fr\'echet algebras}
\author[A.~Acuaviva]{Antonio Acuaviva}
\address{(A.~Acuaviva) School of Mathematical Sciences, Charles Carter Building, Lancaster University, Lancaster LA1 4YX, United Kingdom}
\email{ahacua@gmail.com}
\date{22 September 2026}
\begin{abstract}
We prove that every character on a commutative complex complete Hausdorff locally multiplicatively convex algebra, not necessarily unital, is bounded on bounded sets. Consequently, every character on a commutative complex Fr\'echet--Arens--Michael algebra is continuous, giving an affirmative answer to Michael's problem. In the Fr\'echet case, each character is bounded by one seminorm from any increasing defining sequence of submultiplicative seminorms.
\end{abstract}
\maketitle

\begin{center}
\emph{Companion Lean~4 formalisation:} \url{https://michaels-problem.github.io/}
\end{center}

\section{Introduction}\label{sec:introduction}

A \emph{character} on a complex algebra is a nonzero complex-linear multiplicative functional. Every character $\chi$ on a complex Banach algebra $A$ is continuous, since
\begin{equation*}
  \abs{\chi(a)}\le r_A(a)\le\norm{a}\qquad(a\in A),
\end{equation*}
where $r_A(a)$ denotes the spectral radius. In the unital case, the estimate follows from $\chi(a)\in\sigma_A(a)$, and unitisation gives the general case. This classical result is part of Gelfand's character theory of commutative Banach algebras~\cite[Sections~5--6]{Gelfand1941} and provides a basic example of \emph{automatic continuity}: algebraic conditions on a map imply its continuity. The same theme appears in Johnson's theorem that every surjective algebra homomorphism from a Banach algebra onto a semisimple Banach algebra is continuous~\cite{Johnson1967}. For a broader account of automatic continuity in Banach algebras, see Dales's monograph~\cite{Dales2000}.

Arens~\cite{Arens1952} and Michael~\cite{Michael1952} extended Banach algebra theory to \emph{locally multiplicatively convex algebras}, whose topology is defined by submultiplicative seminorms. A complete Hausdorff algebra of this kind is called an \emph{Arens--Michael algebra}. In the metrisable case, we call it a \emph{Fr\'echet--Arens--Michael algebra}; its topology can be defined by an increasing sequence $(p_n)_{n=1}^{\infty}$ of submultiplicative seminorms, and it is a projective limit of the Banach algebras obtained by completing the normed quotients $A/\ker p_n$.

The spectral argument above does not extend directly to this setting, since spectra need no longer be bounded. For example, the coordinate function $z$ has spectrum $\C$ in the algebra $\mathcal H(\C)$ of entire functions with the topology of uniform convergence on compact sets. Spectral inclusion alone therefore gives no estimate in terms of a fixed defining seminorm, whereas continuity of $\chi$ requires one such estimate to hold for every $a\in A$.

Michael asked in 1952 whether every character on a commutative complex Fr\'echet--Arens--Michael algebra is continuous~\cite[Section~12, Question~1]{Michael1952}. This has come to be known as \emph{Michael's problem}. He also asked whether every character on a complete commutative locally multiplicatively convex algebra maps bounded sets to bounded sets, even without assuming metrizability~\cite[Section~12, Question~2]{Michael1952}. Our main result gives an affirmative answer.

\begin{maintheorem}\label{thm:boundedness}
Let $A$ be a commutative complex Arens--Michael algebra. Every character $\chi\colon A\to\C$ is bounded on bounded subsets of $A$.
\end{maintheorem}

As Michael noted, a bounded linear functional on a metrisable locally convex space is continuous. Theorem~\ref{thm:boundedness} therefore yields automatic continuity in the Fr\'echet case.

\begin{maincorollary}\label{cor:continuity}
Let $A$ be a commutative complex Fr\'echet--Arens--Michael algebra. Every character $\chi\colon A\to\C$ is continuous. More precisely, for every increasing defining sequence $(p_n)_{n=1}^{\infty}$ of submultiplicative seminorms, there is $n\in\N$ such that
\begin{equation*}
  \abs{\chi(a)}\le p_n(a)\qquad(a\in A).
\end{equation*}
\end{maincorollary}

We emphasise that neither Theorem~\ref{thm:boundedness} nor Corollary~\ref{cor:continuity} requires the algebra to be unital. We first prove boundedness for unital algebras and then pass to the general case by unitisation.

An early partial answer follows from Arens's finite joint-spectrum theorem~\cite[Theorem~6.3; see Theorem~6.32 for the Fr\'echet formulation]{Arens1958}: on a commutative unital Fr\'echet--Arens--Michael algebra, the values of any character on a finite tuple can be matched exactly by a continuous character. If the algebra has finitely many topological generators, interpolate on these generators together with an arbitrary additional element. The resulting continuous characters agree on the generators and hence everywhere, proving continuity of the original character. In general, however, the interpolating continuous character may depend on the chosen tuple.

For the general problem, Clayton~\cite{Clayton1975} introduced reductions to particular test algebras. One of these is an algebra of entire power series in countably many variables, described in coefficient form by Dales, Patel and Read~\cite[Definition~9.2]{DalesPatelRead2010}. A related geometric approach was developed by Dixon and Esterle~\cite[Theorem~3.3]{DixonEsterle1986}, who connected the problem with inverse sequences of complex Euclidean spaces and entire bonding maps. Esterle~\cite{Esterle1996} pursued this connection through Mittag-Leffler methods and joint spectra; his survey~\cite{Esterle2012} gives a fuller account of these reductions and their relation to other automatic continuity problems.

Affirmative solutions have also been proposed by Patel~\cite{Patel2022} and Laayouni~\cite[Theorem~2.6]{Laayouni2024}, but the proposed proofs appear to contain gaps. Our proof is independent of the arguments of Patel and Laayouni. It combines the test-algebra reduction with Arens's theorem and Oka approximation, as we now explain.

\subsection*{Idea of the proof and organisation}
The starting point is the coefficient algebra $\E$, defined in Section~\ref{sec:coefficient}, whose continuous characters are evaluations at bounded complex sequences. For $j\in\N$, we write $z_j\in\E$ for the $j$-th coordinate function, given by $z_j(w)=w_j$. If a character $\chi$ on a commutative unital Arens--Michael algebra were unbounded on some bounded set, we could choose a bounded sequence $(x_j)_{j=1}^{\infty}$ with $\chi(x_j)=4j$. Proposition~\ref{prop:reduction} shows that substitution along this sequence gives a character $\Phi$ on $\E$ satisfying $\Phi(z_j)=4j$ for every $j\in\N$. This reduction uses completeness of the original algebra but does not require metrizability.

For $k\in\N$, let $S_k$ be the set of bounded sequences whose first $k$ coordinates are $4,8,\ldots,4k$. Given any $f,g\in\E$, Arens's theorem applied to $z_1,\ldots,z_k,f,g$ would give a point $w\in S_k$ such that
\begin{equation*}
  (f(w),g(w))=(\Phi(f),\Phi(g)).
\end{equation*}
Thus the fixed vector on the right would belong to every $F(S_k)$, where $F=(f,g)$ maps each bounded complex sequence $w$ to $(f(w),g(w))\in\C^2$. It suffices, therefore, to construct one pair $f,g\in\E$ for which these images have empty intersection.

Proposition~\ref{prop:escape}, the central step of the proof, constructs precisely such a pair. Its proof in Section~\ref{sec:escape} uses successive approximation with entire maps $F_N\colon\C^N\to\C^2$ and establishes the stronger estimate
\begin{equation*}
  \norm{F(w)}_2\ge k\qquad(w\in S_k,\ k\in\N).
\end{equation*}
At each stage, we impose the required lower bounds on a finite nested family of affine subspaces. Lemma~\ref{lem:flag-approximation} expresses these conditions as a problem of approximating sections of a holomorphic submersion: the complements of closed balls in $\C^2$ are Oka by Kusakabe's theorem~\cite[Corollary~1.3]{Kusakabe2024}, so Forstneri\v c's theorem~\cite[Theorem~1.4(B)]{Forstneric2010} gives the required holomorphic sections.

We choose the approximation errors so that Cauchy's estimates give convergence of $(F_N)$ in $\E^2$. The limit $F=(f,g)$ inherits all the lower bounds. Section~\ref{sec:completion} combines this construction with finite interpolation to exclude $\Phi$ and proves Theorem~\ref{thm:boundedness} by unitisation. Automatic continuity and the seminorm estimate in Corollary~\ref{cor:continuity} then follow.

\section{The coefficient algebra and the universal reduction}\label{sec:coefficient}

Throughout the paper, all algebras and linear maps are over $\C$, and characters are not assumed continuous or bounded unless this is stated. A subset of a locally convex space is bounded if every continuous seminorm is bounded on it. We work with unital algebras until the final section; every character on such an algebra takes the identity to $1$. We write $\N=\{1,2,\ldots\}$ and $\N_0=\{0,1,2,\ldots\}$.

We begin by describing the coefficient algebra used in the reduction. Let $z=(z_j)_{j=1}^{\infty}$ be a sequence of commuting indeterminates, and let $\M=\N_0^{(\N)}$ be the countable additive monoid of finitely supported sequences of nonnegative integers. For $\alpha=(\alpha_j)_{j=1}^{\infty}\in\M$, write
\begin{equation*}
  \abs{\alpha}=\sum_{j=1}^{\infty}\alpha_j,\qquad z^\alpha=\prod_{j=1}^{\infty}z_j^{\alpha_j}.
\end{equation*}
We consider formal power series $f=\sum_{\alpha\in\M}c_\alpha z^\alpha$ in the commuting variables $(z_j)_{j=1}^{\infty}$. The coefficient algebra $\E$ is the subalgebra consisting of those series whose coefficients satisfy the following weighted absolute summability conditions:
\begin{equation*}
  \E=\left\{\sum_{\alpha\in\M}c_\alpha z^\alpha: q_r(f)=\sum_{\alpha\in\M}\abs{c_\alpha}r^{\abs{\alpha}}<\infty \text{ for every }r\in\N\right\}.
\end{equation*}
Multiplication is the usual product of formal power series, given by coefficient convolution. With this product, $\E$ is the algebra $\mathcal U$ of~\cite[Definition~9.2]{DalesPatelRead2010}. Absolute summation gives
\begin{equation*}
  q_r(fg)\le q_r(f)q_r(g)\qquad(f,g\in\E).
\end{equation*}
The norms $(q_r)_{r=1}^{\infty}$ are increasing, and a sequence that is Cauchy in all of them has a limit in each weighted $\ell^1$ coefficient space. These limits have the same coefficients, so $\E$ is complete and hence a commutative unital Fr\'echet--Arens--Michael algebra. Finite sums of monomials are dense in $\E$, since truncation approximates any element in each specified $q_r$ and the norms are increasing.

To describe the continuous characters on $\E$, let $\ell^\infty$ denote the space of bounded complex sequences with its supremum norm. For $w=(w_j)_{j=1}^{\infty}\in\ell^\infty$ and $f\in\E$, the coefficient series defining $f(w)$ converges absolutely, and for every $R\in\N$ with $R\ge\max\{1,\norm{w}_\infty\}$ we have
\begin{equation}\label{eq:evaluation-bound}
  \abs{f(w)}\le q_R(f).
\end{equation}

The following description of the continuous characters appears in Dixon and Esterle~\cite[Remark~2.2]{DixonEsterle1986} and is also noted in~\cite[p.~138]{DalesPatelRead2010}. We include the short proof for completeness.

\Needspace{7\baselineskip}
\begin{lemma}\label{lem:characters}
The continuous characters on $\E$ are exactly the evaluations at points of $\ell^\infty$.
\end{lemma}
\begin{proof}
If $\eta$ is a continuous character, there are $C>0$ and $r\in\N$ such that $\abs{\eta(f)}\le Cq_r(f)$ for all $f\in\E$. The sequence $w_j=\eta(z_j)$ is bounded, because $q_r(z_j)=r$ for every $j$. On finite polynomials, $\eta$ agrees with evaluation at $w$. Equation~\eqref{eq:evaluation-bound} and density extend this equality to $\E$.

Conversely, evaluation at a bounded sequence is continuous by~\eqref{eq:evaluation-bound}. It takes the identity to $1$ and is multiplicative on finite polynomials, so density and continuity of multiplication show that it is a character on $\E$.
\end{proof}

We next transfer a hypothetical unbounded character to $\E$, with prescribed values on the coordinate functions. The argument is the usual test-algebra substitution, applied to a bounded sequence rather than to a sequence chosen from a countable neighbourhood basis; compare~\cite[Theorem~9.3]{DalesPatelRead2010}.

\begin{proposition}\label{prop:reduction}
If a commutative unital Arens--Michael algebra has a character that is unbounded on a bounded set, then $\E$ has a character $\Phi$ satisfying
\begin{equation*}
  \Phi(z_j)=4j\qquad(j\in\N).
\end{equation*}
\end{proposition}
\begin{proof}
Let $B\subset A$ be bounded and suppose that $\chi(B)$ is unbounded. Choose $b_n\in B$ with $\abs{\chi(b_n)}\ge4n$ for every $n\in\N$, and put
\begin{equation*}
  x_n=\frac{4n}{\chi(b_n)}b_n,\qquad \chi(x_n)=4n.
\end{equation*}
The scalar factors have modulus at most $1$, so $(x_n)_{n=1}^{\infty}$ is bounded in $A$. For each continuous submultiplicative seminorm $p$ on $A$, put
\begin{equation*}
  M_p=\max\left\{1,\sup_{n\in\N}p(x_n)\right\},\qquad R_p=\lceil M_p\rceil.
\end{equation*}
Since $A$ is commutative, the substitution $z_n\mapsto x_n$ defines a unital algebra homomorphism on finite polynomials. For a polynomial $f=\sum_\alpha c_\alpha z^\alpha$, submultiplicativity gives
\begin{equation*}
  p\left(\sum_\alpha c_\alpha x^\alpha\right) \le\max\{1,p(1)\}\,q_{R_p}(f).
\end{equation*}
Here the factor $\max\{1,p(1)\}$ accounts for the constant monomial. Now, for an arbitrary $f\in\E$, approximate it by a sequence of finite polynomials in all the norms $q_r$. Their images form a Cauchy sequence in every defining seminorm of $A$, so completeness gives a limit independent of the approximating sequence. The estimate above thus extends substitution to a continuous linear map $T\colon\E\to A$. Continuity of multiplication shows that $T$ remains a unital algebra homomorphism. Taking $\Phi=\chi\circ T$ therefore proves the proposition.
\end{proof}

\section{Holomorphic escape along affine flags}\label{sec:escape}

Motivated by the coordinate values in Proposition~\ref{prop:reduction}, we define
\begin{equation*}
  S_k=\{w\in\ell^\infty:w_j=4j\text{ for }1\le j\le k\} \qquad(k\in\N).
\end{equation*}
As in the introduction, for $f,g\in\E$ we write $F=(f,g)$ for the map from $\ell^\infty$ to $\C^2$ that sends each bounded sequence $w$ to $(f(w),g(w))$. Although the sets $S_k$ are nonempty and have empty intersection, this alone does not contradict finite interpolation. We need to choose $f$ and $g$ so that the images $F(S_k)$ also have empty intersection. The next proposition is the central step of the proof: it constructs such a pair by making $\norm{F(w)}_2$ uniformly large on each $S_k$.

\begin{proposition}\label{prop:escape}
There exist $f,g\in\E$ such that $F=(f,g)$ satisfies
\begin{equation*}
  \norm{F(w)}_2\ge k\qquad(w\in S_k,\ k\in\N).
\end{equation*}
In particular, $\bigcap_{k=1}^{\infty}F(S_k)=\varnothing$.
\end{proposition}

The proof proceeds by constructing entire maps in finitely many variables and then taking a limit in $\E^2$. Before carrying out these two steps, we establish the approximation lemma needed for the construction.

\subsection*{Approximation with bounds on affine subspaces}
The approximation step uses Oka theory. A complex manifold is called \emph{Oka} if it has the convex approximation property: for every $m\in\N$, holomorphic maps from a neighbourhood of a compact convex set in $\C^m$ can be approximated on that set by entire maps into the manifold. The following lemma combines Kusakabe's result~\cite[Corollary~1.3]{Kusakabe2024} with Forstneri\v c's theorem~\cite[Theorem~1.4(B)]{Forstneric2010} to obtain approximation subject to lower bounds along nested affine subspaces.

\begin{lemma}\label{lem:flag-approximation}
Let $L_1\supsetneq\ldots\supsetneq L_m$ be nonempty proper affine subspaces of $\C^n$, let $0<r_1\le\ldots\le r_m$, and let $K\subset\C^n$ be compact and polynomially convex. Suppose that $h\colon\C^n\to\C^2$ is continuous, holomorphic on a neighbourhood of $K$, and satisfies
\begin{equation*}
  \norm{h(z)}_2>r_j\qquad(z\in L_j,\ 1\le j\le m).
\end{equation*}
Then, for every $\varepsilon>0$, there is an entire map $H\colon\C^n\to\C^2$ with
\begin{equation*}
  \sup_{z\in K}\norm{H(z)-h(z)}_2<\varepsilon, \qquad \norm{H(z)}_2>r_j\quad(z\in L_j,\ 1\le j\le m).
\end{equation*}
\end{lemma}
\begin{proof}
Write $\overline B(r)=\{v\in\C^2:\norm{v}_2\le r\}$, and set
\begin{equation*}
  Z=(\C^n\times\C^2)\setminus\bigcup_{j=1}^m \bigl(L_j\times\overline B(r_j)\bigr).
\end{equation*}
Let $\pi\colon Z\to\C^n$ be the projection onto the first factor, given by $\pi(z,v)=z$. The condition $(z,v)\in Z$ means that $\norm{v}_2>r_j$ whenever $z\in L_j$. Each set $L_j\times\overline B(r_j)$ is closed, so their finite union is closed and $Z$ is open in $\C^{n+2}$. The restriction of the coordinate projection to this open set is a holomorphic submersion.

To describe the bundle structure, set $L_0=\C^n$ and $L_{m+1}=\varnothing$. The chain $(L_j)_{j=0}^{m+1}$ is a finite filtration by closed complex subvarieties, with strata $L_j\setminus L_{j+1}$ for $0\le j\le m$. Each stratum is a smooth complex manifold, since it is an open subset of the affine subspace $L_j$.

If $z\in L_0\setminus L_1$, then $z$ belongs to none of the subspaces $L_1,\ldots,L_m$, so the definition of $Z$ places no restriction on $v$. Hence $\pi^{-1}(L_0\setminus L_1)=(L_0\setminus L_1)\times\C^2$. Now fix $1\le j\le m$ and $z\in L_j\setminus L_{j+1}$. The point $z$ belongs precisely to $L_1,\ldots,L_j$. Thus the part removed from the fibre above $z$ is $\bigcup_{i=1}^{j}\overline B(r_i)=\overline B(r_j)$, where the equality follows from $r_1\le\ldots\le r_j$. Consequently,
\begin{equation*}
  \pi^{-1}(L_j\setminus L_{j+1})=(L_j\setminus L_{j+1})\times\bigl(\C^2\setminus\overline B(r_j)\bigr)\qquad(1\le j\le m).
\end{equation*}
In both cases, $\pi$ is the projection onto the first factor. Its restriction over each stratum is therefore a trivial holomorphic fibre bundle.

Since closed Euclidean balls are polynomially convex, their complements in $\C^2$ are Oka by Kusakabe's theorem~\cite[Corollary~1.3]{Kusakabe2024}; the fibre $\C^2$ is also Oka. Thus the restrictions of $\pi$ to the strata satisfy the bundle hypothesis of~\cite[Definitions~1.2--1.3]{Forstneric2010}. This hypothesis requires local triviality only over each stratum, and imposes no Stein condition on the total space. Choosing $a\in L_m$, the open balls $U_\nu=B(a,\nu)$ exhaust $\C^n$, and intersecting the affine filtration with each $U_\nu$ gives the stratification required by Forstneri\v c's theorem.

The graph of $h$ is a continuous section of $\pi$, holomorphic near $K$. Polynomial convexity implies that $K$ is $\mathcal O(\C^n)$-convex, where $\mathcal O(\C^n)$ denotes the algebra of entire functions on $\C^n$. We may therefore apply Forstneri\v c's theorem~\cite[Theorem~1.4(B)]{Forstneric2010}, with empty interpolation subvariety, to approximate it on $K$ by a holomorphic section. Using the Euclidean distance inherited from $\C^{n+2}$ gives the required estimate for the fibre component $H$, while the condition that its graph lies in $Z$ gives all the lower bounds.
\end{proof}

We now use Lemma~\ref{lem:flag-approximation} to prove Proposition~\ref{prop:escape}.

\begin{proof}[Proof of Proposition~\ref{prop:escape}]
We divide the argument into two steps.

\smallskip
\noindent\emph{Step 1: Construction in finitely many variables.} We construct entire maps $F_N\colon\C^N\to\C^2$, $N\in\N$. When comparing successive maps, we regard $F_{N-1}$ as a map on $\C^N$ independent of its last coordinate. The maps will satisfy
\begin{equation}\label{eq:finite-escape}
  \norm{F_N(z)}_2>k \quad\text{if }z_j=4j\ (1\le j\le k), \qquad 1\le k\le N,
\end{equation}
with no restrictions on the remaining coordinates. To ensure convergence in $\E^2$, we also require, for $N\ge2$,
\begin{equation}\label{eq:approximation}
  \sup_{z\in K_N}\norm{F_N(z)-F_{N-1}(z)}_2<2^{-3N}, \qquad K_N=\{z\in\C^N:\abs{z_j}\le2N\ (1\le j\le N)\}.
\end{equation}

Set $F_1$ to be the constant function with value $(2,0)$, which satisfies~\eqref{eq:finite-escape}. Suppose that $N\ge2$ and $F_{N-1}$ has been constructed. For $1\le k\le N$, put
\begin{equation*}
  L_{N,k}=\{z\in\C^N:z_j=4j\ (1\le j\le k)\}.
\end{equation*}
These subspaces form a strictly decreasing affine flag. To apply Lemma~\ref{lem:flag-approximation} with $L_k=L_{N,k}$ and $r_k=k$, we first modify the preceding map near the new constrained point.

Write
\begin{equation*}
  H_N(z)=F_{N-1}(z_1,\ldots,z_{N-1}),\qquad \xi_N=(4,8,\ldots,4N).
\end{equation*}
The bounds in~\eqref{eq:finite-escape} with $k<N$ already hold for $H_N$, and in particular $\norm{H_N(\xi_N)}_2>N-1$. It remains to obtain the bound at $\xi_N$ while keeping the map unchanged near $K_N$.

Since $4N>2N$, the point $\xi_N$ lies outside $K_N$. Choose a continuous function $b_N\colon\C^N\to[0,1]$ with $b_N(\xi_N)=1$ and compact support in a neighbourhood of $\xi_N$ disjoint from $K_N$. Set
\begin{equation*}
  G_N(z)=(1+b_N(z))H_N(z).
\end{equation*}
The real factor $1+b_N(z)\ge1$ preserves all the bounds for $k<N$, while at the new constrained point we have
\begin{equation*}
  \norm{G_N(\xi_N)}_2=2\norm{H_N(\xi_N)}_2>2(N-1)\ge N.
\end{equation*}
Since $L_{N,N}=\{\xi_N\}$, the map $G_N$ satisfies all the inequalities required in Lemma~\ref{lem:flag-approximation}.

Moreover, $G_N$ agrees with the entire map $H_N$ near the polynomially convex polydisc $K_N$. Lemma~\ref{lem:flag-approximation}, applied with $\varepsilon=2^{-3N}$, therefore gives an entire map $F_N$ satisfying both~\eqref{eq:finite-escape} and~\eqref{eq:approximation}, completing the induction.

\smallskip
\noindent\emph{Step 2: Convergence in the coefficient algebra.} We next pass from the maps $F_N$ to a pair in $\E^2$. Every entire function of finitely many variables represents an element of $\E$: for an entire function $h$ on $\C^N$, Cauchy's estimate on a polydisc of radius $2r$ gives
\begin{equation*}
  q_r(h)\le 2^N\sup_{\abs{z_j}\le2r}\abs{h(z)}<\infty.
\end{equation*}
We may therefore regard both components of each $F_N$ as elements of $\E$ and estimate successive differences in its defining norms.

For $N\ge2$, let $h_N$ be either component of $F_N-F_{N-1}$, and write its Taylor series as $\sum_{\alpha\in\N_0^N}c_{N,\alpha}z^\alpha$. By~\eqref{eq:approximation} and Cauchy's estimate,
\begin{equation*}
  \abs{c_{N,\alpha}}\le2^{-3N}(2N)^{-\abs{\alpha}}.
\end{equation*}
Consequently,
\begin{equation*}
  q_N(h_N)\le2^{-3N}\sum_{\alpha\in\N_0^N}2^{-\abs{\alpha}} =2^{-3N}(1-1/2)^{-N}=2^{-2N}.
\end{equation*}
For every fixed $r$, the same bound holds for $q_r(h_N)$ whenever $N\ge\max\{2,r\}$. By completeness, the series
\begin{equation*}
  F=F_1+\sum_{N=2}^{\infty}(F_N-F_{N-1})
\end{equation*}
converges in $\E^2$. Write $F=(f,g)$; then $F_N\to F$ in $\E^2$.

It remains to check the lower bounds for $F$. Fix $k\in\N$ and $w\in S_k$. Evaluation at $w$ is continuous by~\eqref{eq:evaluation-bound}, so $F_N(w)\to F(w)$, while~\eqref{eq:finite-escape} gives $\norm{F_N(w)}_2>k$ for every $N\ge k$. Passing to the limit yields $\norm{F(w)}_2\ge k$, as required. This passage to the limit uses only continuous evaluations; no discontinuous character is applied to an infinite sum.
\end{proof}

\begin{remark}\label{rem:two-components}
The use of two components in Proposition~\ref{prop:escape} is necessary. Indeed, if a scalar $f\in\E$ satisfied $\abs{f(w)}\ge1$ on $S_1$, its restriction to each affine complex line in $S_1$ would be entire, by~\eqref{eq:evaluation-bound}, and would have bounded reciprocal. Liouville's theorem would then make $f$ constant on $S_1$, since any two points of $S_1$ lie on such a line. A scalar function therefore cannot satisfy the increasing lower bounds on the sets $S_k$.
\end{remark}

\section{Boundedness and automatic continuity}\label{sec:completion}

We can now combine the reduction of Section~\ref{sec:coefficient} with Proposition~\ref{prop:escape} and Arens's finite joint-spectrum theorem.

\begin{proof}[Proof of Theorem~\ref{thm:boundedness}]
Suppose first that $A$ is unital and has a character that is unbounded on some bounded subset. Proposition~\ref{prop:reduction} then gives a character $\Phi$ on $\E$ with $\Phi(z_j)=4j$ for every $j\in\N$. Choose $f,g$ as in Proposition~\ref{prop:escape} and set
\begin{equation*}
  v=(\Phi(f),\Phi(g))\in\C^2.
\end{equation*}
For every $k\in\N$, the finitely generated ideal
\begin{equation*}
  I_k=\sum_{j=1}^k\E(z_j-4j)+\E(f-\Phi(f))+\E(g-\Phi(g))
\end{equation*}
is proper, since it is contained in $\ker\Phi$. Arens's finite joint-spectrum theorem~\cite[Theorem~6.3]{Arens1958} implies that $I_k$ is annihilated by a continuous character $\eta_k$; see also~\cite[Section~6, p.~577]{Esterle1996}. By Lemma~\ref{lem:characters}, $\eta_k$ is evaluation at some $w^{(k)}\in\ell^\infty$. Thus $w^{(k)}\in S_k$ and
\begin{equation*}
  v=(f(w^{(k)}),g(w^{(k)})).
\end{equation*}
Proposition~\ref{prop:escape} now gives $\norm{v}_2\ge k$ for every $k\in\N$, which is impossible. Thus every character is bounded on bounded sets in the unital case.

To treat a nonunital algebra $A$, give its unitisation $A^\#=A\oplus\C$ the multiplication
\begin{equation*}
  (a,\lambda)(b,\mu)=(ab+\lambda b+\mu a,\lambda\mu)
\end{equation*}
and the seminorms $p^\#(a,\lambda)=p(a)+\abs{\lambda}$, where $p$ ranges over a defining family of submultiplicative seminorms on $A$. These make $A^\#$ a commutative unital Arens--Michael algebra. A character $\chi$ on $A$ extends to the character
\begin{equation*}
  \chi^\#\colon A^\#\to\C,\qquad \chi^\#(a,\lambda)=\chi(a)+\lambda.
\end{equation*}
By the unital case, $\chi^\#$ is bounded on bounded sets. The inclusion $A\to A^\#$ is continuous, so its restriction $\chi$ has the same property.
\end{proof}

\begin{proof}[Proof of Corollary~\ref{cor:continuity}]
By Theorem~\ref{thm:boundedness}, the character $\chi$ is bounded on bounded sets. If it were discontinuous, we could choose $x_n\in A$ with $p_n(x_n)\le2^{-n}$ and $\abs{\chi(x_n)}\ge n$ for every $n\in\N$. Then $x_n\to0$, so $\{x_n:n\in\N\}$ would be a bounded set on which $\chi$ is unbounded. This contradiction proves continuity.

We can therefore choose $C>0$ and $n\in\N$ such that $\abs{\chi(a)}\le Cp_n(a)$ for all $a\in A$. Applying this estimate to powers of $a$ and using submultiplicativity, we obtain, for every $m\in\N$,
\begin{equation*}
  \abs{\chi(a)}^m=\abs{\chi(a^m)}\le Cp_n(a^m)\le Cp_n(a)^m.
\end{equation*}
Taking $m$th roots and letting $m\to\infty$ gives the required seminorm bound.
\end{proof}

\begin{remark}
Completeness is essential in both results, as the algebra $c$ of convergent complex sequences shows. Equip $c$ with the seminorms $p_n(a)=\max_{1\le j\le n}\abs{a_j}$ and consider the character $\chi(a)=\lim_{j\to\infty}a_j$. For each $m\in\N$, let $u^{(m)}\in c$ be zero in its first $m$ coordinates and equal to $m$ thereafter. Then $u^{(m)}\to0$ in every $p_n$, whereas $\chi(u^{(m)})=m$. Thus $\chi$ is unbounded on a bounded set and is discontinuous. In this topology, $c$ is incomplete and has completion $\C^{\N}$ with the product topology.
\end{remark}

\par\medskip\noindent\textbf{Acknowledgements.}
The author acknowledges funding from the EPSRC (grant number EP/W524438/1) in support of his PhD studies. He thanks his brother, Pablo Acuaviva, for helpful conversations about AI and mathematics.

\par\medskip\noindent\textbf{AI statement.}
While preparing the survey \emph{Mathematical Discovery in the Wild}~\cite{AcuavivaAcuaviva2026}, the author proposed Michael's problem for our experiments with large language models. Those attempts, and a later attempt with GPT-5.6, were unsuccessful. The proof presented here was essentially produced by GPT 6.0 Astra and then edited by the author.

\end{document}